\documentclass[12pt]{amsart}

\usepackage{amsmath, color} \usepackage{amssymb} \usepackage{amsthm}
\usepackage{amscd}
\usepackage{tikz}
\usetikzlibrary{calc}
\usetikzlibrary{patterns}
\usetikzlibrary{matrix}
\usepgflibrary{arrows}
\usetikzlibrary{arrows}
\usetikzlibrary{decorations.pathreplacing}
\usetikzlibrary{decorations.pathmorphing}
\usepackage{epsfig}
\usepackage{enumerate}

\def\Tor{\operatorname{Tor}}

\DeclareMathOperator{\reg}{reg}

\theoremstyle{plain}
\newtheorem{theorem}{Theorem}[section]

\newtheorem{definition}[theorem]{Definition}

\newtheorem{lemma}[theorem]{Lemma}

\newtheorem{observation}[theorem]{Observation}

\begin{document} 

\title{Regularity of Path Ideals of Gap Free Graphs}
\date{\today}

\author[Banerjee]{Arindam Banerjee}
\address{Department of Mathematics, University of Virginia,
Charlottesville, VA, USA} \email{ab4cb@virginia.edu}


\baselineskip 16pt \footskip = 32pt

\begin{abstract}
In this paper we study the Castelnuovo-Mumford regularity of the path ideals of finite simple graphs. We find new upper bounds for various path ideals of gap free graphs. In particular we prove that the path ideals of gap free and claw graphs have linear minimal free resolutions.
\end{abstract}

\maketitle \markboth{A.Banerjee}
{Bounds on regularity}

\section{Introduction}
\bigskip

 In this work we study the Castelnuovo-Mumford regularity of various path ideals related to the finite simple graphs. For any finite simple graph $G$ and for any integer $t\geq 3$, the $t$-path ideal $I_t(G)$ (or simply $I_t$ if the choice for $G$ is understood) of $G$ is the square free monomial ideal generated by the $t$-paths of $G$ in the polynomial ring on the vertices of $G$ (defined in [5]). Various studies have been done on the regularity of a closely related ideal, the so called edge ideal $I(G)$, which is the ideal generated by the edges of $G$. In $[8]$, Fr\"oberg completely characterized the case where $I(G)$ has minimum possible regularity, the case where it has linear minimal free resolution. In [11], Herzog, Hibi and Zheng proved that if an edge ideal has linear resolution so do its powers. In [13], Nevo studied the ``first'' case where $I(G)$ may not have a linear resolution. He proved that if $G$ is both gap free and claw free then $I(G)$ has regularity less than or equal to $3$ and $I(G)^2$ has linear resolution. The author of this paper in generalized Nevo's result in [3] by proving if $G$ is gap free and cricket free (see definition) then $I(G)^n$ has linear minimal free resolution for every $n\geq 2$. Compare to these, the case of the path ideal seems to be relatively less explored but significant works on the regularity of the path ideals have been done in some recent works (for example in [1], [2] and [4]).\\
 
  Motivated by the above mentioned works on the edge ideal we look for results of the similar flavor in the case of the path ideals. In particular we look for the cases where some condition on $G$ and $I(G)$ guarantees linear resolutions for various path ideals. We prove various results of this type in the cases where there are results about the edge ideals, for example the gap free graphs, the claw free graphs and the cricket free graphs. Our main result is the following (compare this to the work done in [3] or [13]):\\   
  
\begin{theorem}
If $G$ is gap free and claw free then for every $t \geq 3$ the squarefree monomial ideal generated by $t$-paths $I_t(G)$ is either the zero ideal or has linear minimal free resolution.
\end{theorem}

 In order to prove these results, we study various colon ideals related to the $I_{t+1}$, the monomials representing $t$-paths and the edges in $G$. We show that under various assumptions on $G$ these colon ideals take very interesting forms and they behave nicely for regularity. Using these observations and various short exact sequences we arrive at several bounds for $\text{reg}(I_{t+1}(G))$ for different values of $t$.\\
 
 Finally the path ideals of simple graphs can also be visualized as the edge ideals of a hypergraph (we thank H. T\`ai H\`a for bringing this to our notice). People have studied the regularity of the edge ideals of hypergraphs and their relation with the combinatorics of the hypergraph. We refer to [9] and [10] for studies in that area.\\

\bigskip

\section{Preliminaries}

Throughout this paper, we let $G$ be a finite simple graph with vertex set $V(G)$  and set of edges $E(G)$. 
A subgraph $G' \subseteq G$ is called \emph{induced} if $uv$ is an edge of $G'$ whenever $u$ and $v$ are vertices of $G'$ and $uv$ is an edge of $G$.

The \emph{complement} of a graph $G$, for which we write $G^c$, is the graph on the same vertex set in which $uv$ is an edge of $G^c$ if and only if it is not an edge of $G$.  

Finally, let $C_k$ denote the cycle on $k$ vertices, and we let $K_{m, n}$ denote the complete bipartite graph with $m$ vertices on one side, and $n$ on the other. The complement of a cycle is called anticycle.  

\begin{definition}
Let $G$ be a graph.  We say two edges $uv$ and $xy$ form a \emph{gap} in $G$ if $G$ does not have an edge with one 
endpoint in $\{u,v\}$ and the other in $\{x,y\}$.
A graph without gaps is called \emph{gap-free}.  Equivalently, $G$ is gap-free if and only if $G^c$ contains no induced $C_4$.
\end{definition}

Thus, $G$ is gap-free if and only if it does not contain two vertex-disjoint edges as an induced subgraph.  


\begin{definition}
Any graph isomorphic to $K_{1, 3}$ is called a \emph{claw}. 
A graph without an induced claw is called \emph{claw-free}. 
\end{definition}

\begin{definition}
Any graph isomorphic to the graph with set of vertices $\{w_1,w_2,w_3,w_4,\\w_5\}$ and set of edges $\{w_1w_3,w_2w_3,w_3w_4,w_3w_5,w_4w_5\}$ is called a cricket.
A graph without an induced cricket is called \emph{cricket-free}.  As a cricket contains a claw, a claw free graph is also cricket free.\\ 
\end{definition}

\begin{definition}
An edge in a graph is called a whisker if any of its vertices has degree one.
\end{definition}


 If $G$ is a graph without isolated vertices then let $S$ denote the polynomial ring on the vertices of $G$ over some fixed field $K$.  Recall that the \emph{edge ideal} of $G$ is 
\[
I(G) = (xy: xy \text{ is an edge of } G).
\]
The $t-$path ideal of $G$ for $t\geq 3$ is
\[I_t(G)=(x_1x_2....x_t: x_1...x_t \text{ is a path of length t in } G).
\]

\begin{definition}
Let $S$ be a standard graded polynomial ring over a field $K$. The Castelnuovo-Mumford regularity of a finitely generated graded $S$ module $M$, written $\reg(M)$ is given by $$\reg(M):= \max \{j-i|\Tor_{i} (M,K)_j \neq 0 \}$$
\end{definition}

\begin{definition}
Let $I$ be a monomial ideal generated in degree $d$. We say that $I$ is \emph{$k$-steps linear} whenever the minimal free resolution of $I$ over the polynomial ring
is linear for $k$ steps, i.e., $\Tor_{i}^S(I,K)_j = 0$ for all $1\leq i\leq k$ and all $j\ne i+d$. We say $I$ has linear minimal free resolution if the minimal free resolution is $k$-steps linear for all $k \geq 1$.
\end{definition}

 We recall a few well known results. We refer reader to [3] and [6] for reference.

\begin{observation} 
Let $I$ be a monomial ideal generated in degree $d$. Then $I$ has linear minimal free resolution if and only if $\reg (I)=d$.
\end{observation}

\begin{lemma}
Let $I \subseteq S$ be a monomial ideal. Then for any variable $x$, $\reg(I,x)\leq \reg(I)$. In particular if $v$ is a vertex in a graph $G$, then $\reg (I(G-v))\leq \reg((I(G))$.
\end{lemma}

 The following theorem follows from Lemma $2.10$ of [3]:

\begin{lemma}\label{exact}
Let $I \subseteq S$ be a monomial ideal, and let $m$ be a monomial of degree $d$.  Then
\[
\reg(I) \leq \max\{ \reg (I : m) + d, \reg (I,m)\}. 
\]
Moreover, if $m$ is a variable $x$ appearing in I, then $\reg(I)$ is {\it equal} to one of
these terms.
\end{lemma}

 Using this lemma repeatedly we get the following result, which is also a varsion of the Lemma 5.1 of [3]:\\
 
\begin{lemma}
Let $J \subseteq I$ be two monomial ideals in the polynomial ring $S$ and $I$ is generated in degree $d$ by $m_1,...,m_k$. Then $$ \text{reg} (J) \leq \text{max} \{ \text{reg} ( (J:m_i)+(m_1:m_i)+...+(m_{i-1}):(m_i))+d, 1\leq i \leq k, \text{reg}(I) \}$$
\end{lemma}

 Finally the following theorem due to Fr\"oberg (See Theorem $1$ of [8]):

\begin{theorem}
The minimal free resolution of $I(G)$ is linear if and only if the complement graph $G^c$ is chordal.
\end{theorem}
 
  We finally recall a theorem on the regularity of cricket free graphs which is the Theorem 3.4 of [3]:\\
  
\begin{theorem}
If $G$ is gap free and cricket free then $\text{reg }(I(G)) \leq 3$.
\end{theorem}

 The following is the Theorem 2.1 of [3]:\\
\begin{theorem}
Let $G$ be a simple graph with edge ideal
$I(G)$. Then $I(G)$ has a $p$-linear resolution if and only if every induced
cycle in $G^c$ that is not a triangle has length $\geq p + 3$.
\end{theorem}

\bigskip

\section{Regularity of path ideals of gap free graphs}

 In this section we study the regularity of the path ideals and find several upper bounds for them. All along we assume that $G$ is a gap free graph whose $t$-path ideal is denoted by $I_t$ for all $t\geq 3$ and whose edge ideal is denoted by $I$.
 
\begin{lemma}
 If $e=uv$ is a generator of $I$ and $I_3 \neq 0$  then $(I_3:e)$ is generated in degree one. As a consequence it is a prime ideal generated by variables.   
\end{lemma}

\begin{proof}
 Let $m$ be a minimal monomial generator of $(I_3:e)$. So there exists $a,b,c \in V(G)$ with $ab, bc \in E(G)$ such that $abc |  uvm$. If $\{a,b,c\} \cap \{u,v\} = \emptyset $ then $abc |m$. As $G$ is gap free one of $ua,va,ub,vb$ is an edge in $G$. Let $ua$ is an edge. So $ae \in I_3$. Hence $m=a$ as $m$ is minimal. Similarly we show that in the other cases also $m$ has degree one. \\
 
  If $u=b$ and $v\neq a,c$ then $a|m$ and $ae \in I_3$ making $a=m$ by minimality of $m$. If $u=b$ and $v=c$ again $a|m, ae \in I_3$ and $m=a$. If $u=a, v\neq b$ then $be \in I_3$ and $b|m$. Hence by similar argument $m=b$. If $u=a, v=b$ then $ce \in I_3$ and $c|m$. Again by similar argument $m=c$. The other cases follow by symmetry.\\
  
  This completes the proof.\\
\end{proof}

 Next we prove a bound for the regularity of $I_3$ in terms of the regularity of $I$.
 
\begin{theorem}
If $I_3 \neq 0$ and $\text{reg} (I)=r$ then $\text{reg} (I_3) \leq \text{max} \{r,3\}$.
\end{theorem}

\begin{proof}
 Notice that for any two different edges $e=ab,f=cd$ with no common vertices, $(e:f)=(e)$. As $G$ is gap free at least one of the vertices of $e$ forms an edge with a vertex of $f$. Without loss of generality we can assume $ac$ is an edge. However we observe that in this case $(a) \subseteq (I_3:f)$.\\
 
 In case $e$ and $f$ has a common vertex, $(e:f)$ is generated by a variable. So it follows from the previous lemma that for different edges $e_1,...,e_k$, $(I_3, e_1,....,e_{k-1}):(e_k)$ is $J$ where $J$ is an ideal generated by some variables.\\
 
 In light of these and the Lemma 2.8 we observe that $\text{reg}((I_3, e_1,....,e_{k-1}):(e_k)) \leq 3$. The result then follows from the Theorem 2.11.
\end{proof}

 The next result bounds $\text{reg} (I_4)$ in terms of $\text{reg} (I)$ in two different cases.
 
\begin{lemma}
 Let $I_4 \neq 0$. For any edge $e=xy$, $(I_4:e)$ is a squarefree quadratic monomial whose minimal monomial generaors are either edges of $G$ which does not share a common vertex with $e$, or square free quadratic monomials $uv$ such that $ua$ and $vb$ are edges in $G$.
\end{lemma}

\begin{proof}
 Clearly any minimal generator has to have degree at least two. Any edge that has no vertex in common with $e$ is a generator of $(I_4:e)$ by the fact that $G$ is gap free. If $m$ is a minimal monomial generator of $(I_4:e)$ then $m$ is not divisible by $x$ or $y$. If $m$ is not divisible by an edge that is that does not have a common vertex with $e$ then $m$ must be divisible by $uv$ for some $u$, which is a neighbor of $x$ and for some $v$, which is a neighbor of $y$, with $u \neq v$. The first part of the theorem follows due to the fact that any such $uv$ is a generator of $(I_4:e)$.\\\\
\end{proof}

   As $(I_4:e)$ is a square free quadratic monomial ideal, it is an edge ideal and we denote the corresponding graph by $G'$. Next two lemmas show that the induced cycles of length greater than or equal to four of $G'^c$ are also induced cycles of $G^c$.\\
 
\begin{lemma} If $G'$ is the graph associated to $(I_4:e)$ then $G'$ is gap free
\end{lemma}

\begin{proof} 
Two edges coming from $G$ can't form a gap in $G'$ as for any two vertex $a$ and $b$ of $G$ which remains a vertex in $G'$, if $ab$ is an edge in $G$ it remains an edge in $G'$ by the construction of $G'$. If $ab$ is an edge in $G$ that remains an edge in $G'$, as $G$ is the graph free either $a$ or $b$ is neighbor of either $x$ or $y$. Hence $ab$ cannot form a gap in $G'$ with any new edge. If $uv$ and $u'v'$ are two new edge in $G'$ with $u,u'$ neighbor of $x$ in $G$ and $v,v'$ neighbor of $y$ in $G$ we observe $uv'$ is an edge in $G'$ and hence we conclude no two new edges can form a gap.\\
\end{proof}
 
\begin{lemma} 
If $G'$ is the graph associated to $(I_4:e)$ then any induced cycle of length greater than or equal to five in $G^c$ is an induced cycle in $G^c$
\end{lemma}

\begin{proof} 
We show that if $w_1...w_n$ is an induced cycle in $G'^c$ with $n\geq 5$ then it is an induced cycle in $G^c$ too. Observe that it is enough to prove that for all $i,j$, $w_i, w_{i+j}$ is not an edge in $E(G')\setminus E(G)$. Suppose on the contrary such $i,j$ exists. Without loss of generality we may choose $j$ to be minimal such that for some $i$, $w_i$ and $w_{i+j}$ are connected in $G'$ but not in $G$. Observe that $j\geq 2$ as $w_i w_{i+1}$ can't be connected in an anticycle. Without loss of generality we may further assume $w_1$ is connected to $x$ and $w_{1+j}$ is connected to $y$ in $G$ and they are not connected to each other in $G$ . Now observe $w_{2+j}$ is not connected to $x$ by an edge in $G$ as that will force $w_{1+j}$ and $w_{2+j}$ to be connected in $G'$  leading to a contradiction. So there exists a smallest $l \geq 0$, $2+j \leq n-l \leq n$ such that $w_{n-l}$ is not connected to $x$ by an edge  in $G$. If $l=0$, then $w_{n}$ is not connected to $x$ by an edge in $G$ and if $l>0$ then $w_{n-l}$ is not connected to $x$ by an edge to $x$  in $G$ and $w_n, w_{n-1},..,w_{n-l+1}$ are connected to $x$ by an edge in $G$\\
 
   Next, we look at the edge $w_2 w_{n-l}$ in $G'$. If $w_2$ is connected to $x$ in $G$ then as $w_{1+j}$ is connected to $y$ that will violate the minimality of $j$. If $w_2$ is connected to $y$ in $G$ then $w_1 w_2$ has to be an edge in $G'$, which will contradict the fact $w_1....w_n$ is an anticycle. We observe $w_{n-l}$ can't be connected to $x$ by selection. If $w_{n-l}$ is connected to $y$ and $l=0$  then  $w_1$ and $w_n$ have to be connected to each other in $G'$. If $w_{n-l}$ is connected to $y$ and $l>0$ then  $w_{n-l+1}$ and $w_{n-l}$ have to be connected to each other in $G'$. Both cases lead to a contradiction as $w_1....w_n$ is an anticycle, so $w_2$ and $w_{n-l}$ are not connected to each other in $G$ and neither of them are connected to $x$ or $y$ (and hence $w_2,w_{n-l},x,y$ are four distinct vertices). As $xy$ is an edge in $G$, $w_2w_{n-l}$ can not be an edge in $G$; otherwise they will form a gap. So $w_2$ is connected to $x$ and this gives a contradiction. Hence $w_1...w_n$ is an induced cycle in $G^c$.\\
\end{proof}

\begin{theorem}
Let $I_4 \neq 0$. If $I$ has a minimal free resolution which is linear up to step $p \geq 2$ then so does $I_4$. In particular if $I$ has linear resolution so does $I_4$.
\end{theorem}

\begin{proof}
 Let $e$ be any edge in $G$ and $G'$ be the graph associated with $(I_4:e)$. By previous lemma and the Lemma 2.12 $G'^c$ does not have an induced cycle of length less than $p+3$ that is not a triangle. Hence $(I_4:e)$ has a minimal free resolution that is linear up to step $p$.\\
 
 Next we observe as $G$ is gap free, for any two different edges $e$ and $f$ in $G$, who does not share a common vertex $(f:e) =(f) \subseteq (I_4:e)$. Hence either $(f:e)$ is generated by a variable or it is contained in $(I_4:f)$. So for different edges $e_1,...,e_k,e$ of $G$, $(I_4,e_1,...,e_k):(e)$ is $(I_4:e)+J$ where $J$ is an ideal generated by some variables.\\
 
  Now let $G'$ be the graph of the $(I_4:e)$ and the $G''$ be the graph obtained from $G$ by deleting the vertices in $J$. As $G'^c$ does not have an induced cycle of length less than $p+3$ that is not a triangle, $G''^c$ does not have any such cycle too. Hence $\beta_{ij}((I_4:e)+J)=\beta_{ij}(I(G''))=0$ for every pair $(i,j)$ such that $j\neq i+2$ and $i\leq p$. Now to finish the proof we observe the following sequence of short exact sequences, here we assume $E(G)=\{e_1,...,e_l\}$ :
  $$0\longrightarrow \frac{S}{(I_4:e_1)} (-2) \overset{.e_1} \longrightarrow \frac{S}{I_4} \longrightarrow \frac{S}{(I_4,e_1)} \longrightarrow 0$$
  $$0\longrightarrow \frac{S}{((I_4,e_1):(e_2))} (-2) \overset{.e_2} \longrightarrow \frac{S}{(I_4,e_1)} \longrightarrow \frac{S}{(I_4,e_1,e_2)} \longrightarrow 0$$
  $$\vdots$$
  $$0\longrightarrow \frac{S}{((I_4,e_1,...,e_{l-1}):(e_l))} (-2) \overset{.e_l} \longrightarrow \frac{S}{(I_4,e_1,...,e_{l-1})} \longrightarrow \frac{S}{I} \longrightarrow 0$$\\
  
  As $\beta_{ij}((I_4,e_1,...,e_{k-1}):(e_k))$ is zero for all $i\leq p$ and $j\neq i+2$, from these short exact sequences we conclude that $\beta_{ij}(I_4)=0$ for all pair $(i,j)$ where $i \leq p$ and $j \leq i+4$. But $I_4$ is a homogeneous monomial ideal generated in degree four so $\beta_{ij}(I_4)=0$ if $j <4$. Hence $\beta_{ij}(I_4)=0$ for all pair $(i,j)$ where $i \leq p$ and $j \leq i+4$. This proves the result.\\
\end{proof}
   
\begin{theorem}
If $G$ is gap free and cricket free then $I_4$ has linear minimal free resolution.
\end{theorem}

\begin{proof}  
    We observe that since $G$ is gap free, for any two different edges $e$ and $f$ in $G$, who does not share a common vertex $(f:e) =(f) \subseteq (I_4:e)$. Hence either $(f:e)$ is generated by a variable or it is contained in $(I_4:f)$. So for different edges $e_1,...,e_k,e$ of $G$, $(I_4,e_1,...,e_k):(e)$ is $(I_4:e)+J$ where $J$ is an ideal generated by some variables. Hence in light of Lemmma 2.10 and Theorem 2.12 it is enough to show that for every edge $e$ the $\text{reg}(I_4:e) \leq 2$ that is if $G'$ is the graph associated with $(I_4:e)$ then $G'^c$ is chordal. \\\ 
    
   We know from Lemma 3.4 that $G'$ is gap free. If $w_1....w_n$ is an induced cycle in $G'^c$ with $n\geq 5$, then it is also an induced cycle in $G^c$ by Lemma 3.5. Then either $w_1$ or $w_3$ is a neighbor of $x$ or neighbor of $y$ else $w_1 w_3$ and $e$ forms a gap in $G$, a contradiction. Without loss of generality, we may assume $w_1$ is a neighbor of $x$. Now every neighbor of $y$ is connected to every neighbor of $x$ in $G'$ if they are not same . As neither $w_1 w_n$, nor $w_1 w_2$ is an edge in $G'$, neither of $w_2$ and $w_n$ are neighbors of $y$ in $G$. So one of them has to be neighbor of $x$, as $G$ is gap free. Again, without loss of generality, we may assume $w_2$ is a neighbor of $x$. Next we consider $w_3 w_n$. As $w_1$ and $w_2$ are neighbors of $x$ and neither $w_1w_n$ nor $w_2w_3$ are edges in $G'$, so neither $w_3$ nor $w_n$ can be neighbor of $y$. Then either $w_3$ or $w_n$ of them has to be a neighbor of $x$. Without loss of generality we may assume $w_3$ is a neighbor of $x$. Notice that $y$ is not connected to $w_1$ in $G$ as that will force $w_2$, a neighbor of $x$ to be connected to $w_1$ in $G'$ leading to a contradiction. Hence $\{y,w_2,x,w_1,w_3 \}$ forms a cricket leading to contradiction. \\
  
  Hence by Theorem 2.11 $\text{reg } (I_4:e)=2$ and our result follows from theorem 2.10.\\
\end{proof}

 Next two lemmas will lead us to the proof of the fact that for any gap free and claw free graph any path ideal has linear minimal free resolution .
 
\begin{lemma}
Let $G$ be gap free and claw free and $I_{t+1} \neq 0$ for some $t$. If $e \neq f$ are two generators of $I_t$  then either $(e:f)$ is generated by a variable or $(e:f) \subseteq (I_{t+1} : f)$.
\end{lemma}

\begin{proof}
  Assume $(e:f)$ is not generated by a variable. That means $m=\frac{e}{\text{gcd} (e,f)}$ is a monomial of degree greater than or equal to $2$. Let $f=x_1....x_t$. First observe that if $a$ is a variable such that $a|m$ and $a x_i$ is an edge in $G$ for any $i \in \{1,2,t-1,t\}$. Then $af \in I_{t+1}$ as $G$ is claw free. This is clear if $ax_1$ or $ax_t$ is an edge as in that case $ax_1....x_t$ or $ax_t....x_1$ will be in $I_{t+1}$. If $ax_2$ is an edge then for $ax_2 x_1 x_3$ to avoid being a claw either $ax_1$ or $ax_3$ or $x_1 x_3$ is an edge. In first case it is again clear. In the second case $a f \in I_{t+1}$ as $x_1 x_2 a x_3...x_t$ forms a $t+1$ path. In third case $ax_2 x_1 x_3...x_t$ forms a $t+1$ path. The other cases follow by symmetry. In all these cases $(e:f) \subseteq (a) \subseteq (I_{t+1}:f)$.\\
  
   If there is an edge $h|m$ then as $G$ is gap free considering $x_1x_2$ and $h$ we get that there is a variable $a$ dividing $m$ such that $ax_1$ or $ax_2$ is an edge and hence we are done by arguments of previous paragraph; also if there exists a variable $a$ dividing $m$ such that for some $i$ both $ax_i$ and $ax_{i+1}$ are edges in $G$ then $x_1...x_iax_{i+1}...x_t$ is a generator of $I_{t+1}$ and $(e:f) \subseteq (a) \subseteq (I_{t+1}:f)$.\\
   
   So we may assume neither of the above holds. As $e| mf$, there exists two variables $a,b$ such that there exists $i,i',j,j', i\neq i', j \neq j', i<i',j<j',i<j$ with $a x_i,ax_{i'}, b x_j,bx_{j'}$ are edges in $G$. As before if $i,i',j$ or $j'$ belongs to $\{1,2,t-1,t\}$ the result follows. If $i=j$, for $x_{i-1}, x_i,a,b$ to avoid being a claw we must have either $ax_{i-1}$ or $bx_{i-1}$ is an edge in $G$ contradicting the assumption; as no edge divides $m$, $ab$ can't be an edge. Arguing similarly we may assume $\{i,i'\} \cap \{j,j'\} = \emptyset$.\\
   
   We may assume $x_1$ or $x_t$ is not connected to any vertex from $\{x_i,x_{i'},x_j,x_j'\}$. If $x_1 x_i$ is an edge in $G$ then $ax_i x_1..x_{i-1}x_{i+1}..x_t$ is a $t+1$ path and we are done. The other cases also follow similarly. Since $G$ is claw free both $x_{i-1} x_{i+1}$ and $x_{j-1} x_{j+1}$ are edges in $G$ for $a,x_i,x_{i-1},x_{i+1}$ and $b,x_j,x_{j-1},x_{j+1}$ to avoid being claw respectively. Similarly one may assume $x_{i'-1}x_{i'+1}$ and $x_{j'-1}x_{j'+1}$ are edge in $G$. As $G$ is gap free considering $ax_i$ and $x_1x_2$ we may assume $x_2x_i$ is an edge. Arguing similarly we may assume $x_l x_{l'}$ is an edge for every $l \in \{2,t-1\}$ and every $l' \in \{i,i',j,j'\}$. As $G$ is claw free for $x_1, x_2, x_i, x_{j'}$ to avoid being a claw $x_i x_{j'}$ has to be an edge in $G$; otherwise $x_1x_i$ or $x_1 x_{j'}$ will be an edge. Since $i < j <j'$, $x_{j'} \neq x_{i+1}$.Considering $a,x_i,x_{i+1},x_{j'}$ one concludes $x_{j'} x_{i+1}$ is an edge.\\
 
 If $i<i'<j'$ then $x_1....x_iax_{i'}x_{j'}x_{i+1}x_{i+2}...x_{i'-1}
 x_{i'+1}....x_{j'-1}x_{j'+1}...x_t$ forms a path and
if $i<j'<i'$ then $x_1....x_iax_{i'}x_{j'}x_{i+1}x_{i+2}...x_{j'-1}
 x_{j'+1}....x_{i'-1}x_{i'+1}...x_t$ forms a path. Hence $a$ is a generator of $(I_{t+1}:e)$ in both cases.\\

\end{proof}

\begin{lemma}
Let $G$ be gap free and claw free and $I_t \neq 0$. If $e$ is a generator of $I_t$ for any $t$ then $(I_{t+1} :e)$ is generated by variables. 
\end{lemma}

\begin{proof}
 Let $m$ be a minimal monomial generator of $(I_{t+1}:e)$. So $me \in I_{t+1}$ and for every $m'|m, m' \neq m$, $m'e \notin I_{t+1}$. Hence there is a $t+1$ path $f$ such that $f|em$.\\
 
 Let $f=x_1....x_{t+1}$. First observe that if $a$ is a variable such that $a|m$ and $a x_i$ is an edge in $G$ for any $i \in \{1,2,t,t+1\}$. Then $af \in I_{t+1}$ as $G$ is claw free. This is clear if $ax_1$ or $ax_{t+1}$ is an edge as in that case $ax_1....x_t$ or $ax_{t+1}....x_2$ will be in $I_{t+1}$. If $ax_2$ is an edge then for $ax_2 x_1 x_3$ to avoid being a claw either $ax_1$ or $ax_3$ or $x_1 x_3$ is an edge. In first case it is again clear. In the second case $a f \in I_{t+1}$ as $x_1 x_2 a x_3...x_t$ forms a $t+1$ path. In third case $ax_2 x_1 x_3...x_t$ forms a $t+1$ path. This forces $m=a$ The other cases follow by symmetry.\\
   If there is an edge $h|m$ then as $G$ is gap free considering $x_1x_2$ and $h$ we get that there is a variable $a$ dividing $m$ such that $ax_1$ or $ax_2$ is an edge and hence we are done by arguments of previous paragraph; also if there exists a variable $a$ dividing $m$ such that for some $i$ both $ax_i$ and $ax_{i+1}$ are edges in $G$ then $x_1...x_iax_{i+1}...x_t$ is a generator of $I_{t+1}$ and we are done.\\
 
  So we may assume neither of the above holds. As $f|me$, either $m$ has degree one in which case we are done or there exists two variables $a,b$ such that there exists $i,i',j,j', i\neq i', j \neq j', i<i',j<j',i<j$ with $a x_i,ax_{i'}, b x_j,bx_{j'}$ are edges in $G$. As before if $i,i',j$ or $j'$ belongs to $\{1,2,t,t+1\}$ the result follows. If $i=j$, for $x_{i-1}, x_i,a,b$ to avoid being a claw we must have either $ax_{i-1}$ or $bx_{i-1}$ is an edge in $G$; as no edge divides $m$, $ab$ can't be an edge contradicting the assumpion. Arguing similarly we may assume $\{i,i'\} \cap \{j,j'\} = \emptyset$.\\	
 
    We may assume $x_1$ or $x_t$ is not connected to any vertex from $\{x_i,x_{i'},x_j,x_j'\}$. If $x_1 x_i$ is an edge in $G$ then $ax_i x_1..x_{i-1}x_{i+1}..x_t$ is a $t+1$ path and we are done. The other cases also follow similarly. Since $G$ is claw free both $x_{i-1} x_{i+1}$ and $x_{j-1} x_{j+1}$ are edges in $G$ for $a,x_i,x_{i-1},x_{i+1}$ and $b,x_j,x_{j-1},x_{j+1}$ to avoid being claw respectively. Similarly one may assume $x_{i'-1}x_{i'+1}$ and $x_{j'-1}x_{j'+1}$ are edge in $G$. As $G$ is gap free considering $ax_i$ and $x_1x_2$ we may assume $x_2x_i$ is an edge. Arguing similarly we may assume $x_l x_{l'}$ is an edge for every $l \in \{2,t\}$ and every $l' \in \{i,i',j,j'\}$. As $G$ is claw free for $x_1, x_2, x_i, x_{j'}$ to avoid being a claw $x_i x_{j'}$ has to be an edge in $G$; otherwise $x_1x_i$ or $x_1 x_{j'}$ will be an edge. Since $i < j <j'$, $x_{j'} \neq x_{i+1}$.Considering $a,x_i,x_{i+1},x_{j'}$ one concludes $x_{j'} x_{i+1}$ is an edge.\\
 
 If $i<i'<j'$ then $x_1....x_iax_{i'}x_{j'}x_{i+1}x_{i+2}...x_{i'-1}
 x_{i'+1}....x_{j'-1}x_{j'+1}...x_t$ forms a path and
if $i<j'<i'$ then $x_1....x_iax_{i'}x_{j'}x_{i+1}x_{i+2}...x_{j'-1}
 x_{j'+1}....x_{i'-1}x_{i'+1}...x_t$ forms a path. Hence $a$ is a generator of $(I_{t+1}:e)$ in both cases.\\
 
\end{proof}

 The main theorem follows from these two lemmas.
 
\begin{theorem}
Let $t$ be an integer greater than two. If $G$ is gap free and claw free and $I_t \neq 0$ then $I_t$ has linear minimal free resolution. 
\end{theorem}

\begin{proof}
 For $t=3$ this follows the Lemma 3.2 and the Theorem 2.12 as a claw free graph is automatically cricket free. Let us assume by induction the result holds for $(t-1)$ for some $t \geq 4$. If $m_1,....,m_k$ are $k$ different monomials representing $(t-1)$-paths then by previous two lemmas $(I_t,m_1,...,m_{k-1}):(m_k)$ is an ideal generated by variables and hence has regularity $1$. The result now follows from the Lemma 2.10.
\end{proof}

\textbf{Acknowledgements.} The author is very grateful to his advisor C. Huneke for constant support, valuable ideas and suggestions throughout the project. The author thanks A. Alilooee for valuable discussions. The author is greatful to H.T. H\`a and A. Van Tuyl for valuable suggestions.

\end{document}